\documentclass[a4paper]{article}

\usepackage{amsmath}
\usepackage{amssymb}
\usepackage{array}
\usepackage[table]{xcolor}
\usepackage{breqn}
\usepackage{listings} 
\usepackage{amsthm}
\usepackage{mathrsfs}
\usepackage{enumerate} 
\usepackage{changepage}
\usepackage{tabularx}
\usepackage{pdflscape}
\usepackage{afterpage}
\usepackage{longtable}
\usepackage{mathtools} 
\usepackage{floatpag}
\usepackage{authblk}
\usepackage[numbers]{natbib}

\newcounter{Steps}
\title{{\huge Constructing Majorana Representations}}
\date{}

\usepackage{authblk}

\author[1]{Markus Pfeiffer}
\affil[1]{\small{School of Computer Science, University of St Andrews, KY16 9SX markus.pfeiffer@st-andrews.ac.uk}}

\author[2]{Madeleine Whybrow}
\affil[2]{\small{Department of Mathematics, Imperial College London, SW7 2AZ madeleine.whybrow10@imperial.ac.uk}}

\newtheorem{thm}{Theorem}[section]
\newtheorem{defn}[thm]{Definition}
\newtheorem{lem}[thm]{Lemma}
\newtheorem{prop}[thm]{Proposition}

\newcommand{\ml}[2]{\begin{tabular}{@{} >{$}#1<{$} @{}} #2 \end{tabular}}

\usepackage{natbib}

\usepackage[top=1in, bottom=1.25in, left=1.25in, right=1.25in]{geometry}

\makeatletter
\newcommand{\thickhline}{%
    \noalign {\ifnum 0=`}\fi \hrule height 1pt
    \futurelet \reserved@a \@xhline
}
\newcolumntype{"}{@{\hskip\tabcolsep\vrule width 1.25pt\hskip\tabcolsep}}
\makeatother

\floatpagestyle{empty}

\begin{document}

\begin{center}
    \begin{minipage}{13cm}
        \maketitle 
    \end{minipage}
\end{center}

\begin{abstract}
Majorana theory was introduced by A.~A.~Ivanov as an axiomatic framework in which to study objects related to the Monster simple group and the Griess algebra. Since its inception, it has been used to construct a number of new and important subalgebras of the Griess algebra. The objects at the centre of the theory are known as Majorana algebras and can be studied either in their own right, or as Majorana representations of finite groups. 

In this paper, we present an algorithm to construct the Majorana representations of a given group. We also list a number of groups for which we have constructed Majorana representations, including some which give new examples of Majorana algebras. This work is inspired by that of \'{A}.~Seress.
\end{abstract}

\section{Introduction}

In \cite{Sakuma07}, S.~Sakuma reproved certain important properties of the Monster simple group and the Griess algebra in the framework of vertex operator algebras. Inspired by this result, A.~A.~Ivanov introduced Majorana theory in \cite{Ivanov09} as an axiomatisation of certain properties of the $2A$ axes of the Griess algebra $V_{\mathbb{M}}$. 

The objects at the centre of the theory are \emph{Majorana algebras}, real, commutative, non-associative algebras generated by idempotents known as \emph{Majorana axes}. To each Majorana axis, we can associate a unique involution in the automorphism group of the algebra, known as a \emph{Majorana involution}. The link between Majorana axes and their corresponding involutions forms a strong relationship between Majorana algebras and groups. In particular, most Majorana algebras have been constructed as \emph{Majorana representations} of finite groups. 

One of the main aims of the theory is to provide a framework in which to construct large subalgebras of the Griess algebra. In particular, suppose that, for a given group $G$, there exists an embedding $\iota$ of $G$ into $\mathbb{M}$ and a generating set $T$ of involutions of $G$ such that $\iota(T)$ is contained in the $2A$ class of involutions of $\mathbb{M}$. Then $G$ must admit a Majorana representation that is isomorphic to the subalgebra of $V_{\mathbb{M}}$ generated by the $2A$ axes corresponding to the $2A$ involutions $\iota(T)$. 

Many important subalgebras of the Griess algebra have been constructed in this way, including algebras corresponding to the groups $S_4$ \cite{IPSS10}, $A_5$ \cite{ISe12} and $L_3(2)$ \cite{ISh12}. However, a crucial turning point came in 2012 when \'{A}.~Seress \cite{Seress12} announced the implementation of an algorithm to construct the Majorana representations of a given group. He also gave a list of groups for which he had constructed Majorana representations including many new and important examples such as the groups $L_2(11)$ and $M_{11}$. 

Sadly Seress passed away shortly after the publication of this paper. As far as we are aware, neither his code, nor the full details of the representations which he constructed were ever recovered. As such, reconstructing his algorithm and repeating his results has been an important aim of the theory ever since.

We have used GAP \cite{GAP16} to implement and run an algorithm to construct Majorana representations and have completely recovered Seress' results. Our algorithm is based largely on Seress' methods but we have made some improvements in the implementation. In particular, certain changes have enabled us to extend the algorithm to be able to construct generic representations, rather than just the $2$-closed examples. 

We first present some basic results and definitions from Majorana theory before detailing the implementation of the algorithm. Finally, we give some details of the certain representations which we have constructed, including two new $3$-closed algebras. All of our work can be found at \texttt{https://github.com/mwhybrow92/MajoranaAlgebras}.

\section{Majorana theory}

Let $V$ be a real vector space equipped with an inner product $(,)$ and a commutative, non-associative product $\cdot$ such that
\begin{description}

\item[M1] $( \, , \, )$  associates with $\cdot$ in the sense that 
\[
(u , v \cdot w ) = (u \cdot v, w ) 
\]
for all $u,v,w \in V$;

\item[M2] the Norton Inequality holds so that  
\[
(u \cdot u, v \cdot v) \geq (u \cdot v, u \cdot v)
\]
for all $u,v \in V$.

\end{description} 

Let $A$ be a subset of $V \backslash \{0\}$ and suppose that for every $a \in A$ the following conditions ((M3) to (M7)) hold. 

\begin{description}

\item[M3] $(a,a) = 1$ and $a \cdot a = a$, so that the elements of $A$ are idempotents of length 1;

\item[M4] $V = V_1^{(a)} \oplus V_0^{(a)} \oplus V_{\frac{1}{2^2}}^{(a)} \oplus V_{\frac{1}{2^5}}^{(a)}$ where $V_{\mu}^{(a)} = \{ v \, : \, v \in V, \, a \cdot v = \mu v\}$ is the set of $\mu$-eigenvectors of the adjoint action of $a$ on $V$;

\item[M5] $V_1^{(a)} = \{ \lambda a \, : \, \lambda \in \mathbb{R} \}$;

\item[M6] the linear transformation $\tau(a)$ of $V$ defined via 

\[
\tau(a): u \mapsto (-1)^{2^5 \mu}u
\]

for $u \in V_{\mu}^{(a)}$ with $\mu \in \{ 1,0,\frac{1}{2^2}, \frac{1}{2^5}\}$, preserves the algebra product (i.e. $u^{\tau(a)} \cdot v^{\tau(a)} = (u \cdot v)^{\tau(a)}$ for all $u, v \in V$);

\item[M7] if $V_+^{(a)}$ is the centraliser of $\tau(a)$ in $V$, so that $V_+^{(a)} = V_1^{(a)} \oplus V_0^{(a)} \oplus V_{\frac{1}{2^2}}^{(a)}$, then the linear transformation $\sigma(a)$ of $V_+^{(a)}$ defined via 
\[
\sigma(a) : u \mapsto (-1)^{2^2 \mu}u
\]  
for $u \in V_{\mu}^{(a)}$ with $\mu \in \{ 1,0,\frac{1}{2^2} \}$, preserves the restriction of the algebra product to $V_+^{(a)}$ (i.e. $u^{\sigma(a)} \cdot v^{\sigma(a)} = (u \cdot v)^{\sigma(a)}$ for all $u,v \in V_+^{(a)}$. 

\end{description}

\begin{defn}
\label{defn:majorana}
The elements of $A$ are called \emph{Majorana axes} while the automorphisms $\tau(a)$ are called \emph{Majorana involutions}. A real, commutative, non-associative algebra $(V, \cdot, (\, , \, ))$ is called a \emph{Majorana algebra} if it satisfies axioms (M1) and (M2) and is generated by a set of Majorana axes.
\end{defn}

A Majorana algebra $V$ generated by a set $A$ of Majorana axes is invariant under the group of automorphisms generated by the Majorana involutions $\tau(a)$ for $a \in A$. This idea is axiomatised in the following definition. 

\begin{defn}
\label{defn:majoranarep}
A \emph{Majorana representation} is a tuple
\[
\textbf{R} = (G,T,V,\cdot, (,), \varphi, \psi)
\]
where 
\begin{itemize}
	\item $G$ is a finite group;
	\item $T$ is a $G$-invariant set of generating involutions of $G$;
	\item $V$ is a real vector space equipped with an inner product $(\, , \, )$ and bilinear, commutative, non-associative algebra product $\cdot$ satisfying (M1) and (M2) such that $V$ is generated by a set $A$ of Majorana axes;
	\item $\varphi: G \rightarrow GL(V)$ is a linear representation that preserves both products (i.e. for all $g \in G$ and for all $u, v \in V$, $(u^{\varphi(g)}, v^{\varphi(g)}) = (u,v)$ and $u^{\varphi(g)} \cdot v^{\varphi(g)} = (u \cdot v)^{\varphi(g)}$);
	\item $\psi: T \rightarrow A$ is a bijective mapping such that
	\[
	\psi(t^g) = \psi(t)^{\varphi(g)}.
	\] 
\end{itemize}
\end{defn}
Later in the text, we usually denote a Majorana representation by the triple $(G,T,V)$ wherever it is clear which products and maps are being used.
\begin{defn}
 Let $V$ be a Majorana algebra and let $A$ be the Majorana axes of $V$. For $X \subseteq V$, we denote by $\langle X \rangle$ the subspace of $V$ generated by $X$ and by $\langle \langle X \rangle \rangle$ the subalgebra of $V$ generated by $X$. Moreover, if $X \subseteq A$ then we say that the subalgebra $\langle \langle X \rangle \rangle$ is $k$-\emph{closed} if it is the linear span of the set
\[
\{x_1 \cdot x_2 \cdot \ldots \cdot x_k \, : \, x_i \in X \}.
\]
\end{defn}

In the previous definition, as the algebra product is not necessarily associative, we must specify that the notation $x_1 \cdot x_2 \cdot \ldots \cdot x_k$ refers to the algebra products obtained from \emph{all} possible combinations of brackets on the elements. For example, if $k = 3$, we include $x_1 \cdot (x_2 \cdot x_3)$ as well as $(x_1 \cdot x_2) \cdot x_3$. 

\begin{defn} 
The \emph{nullspace} of a Majorana algebra $V$ is the ideal
\[
N(V) := \{u \in V \mid u\cdot v = 0, \; \forall v \in V\}.
\]  
\end{defn}

The following results can be deduced from axioms (M1) to (M5).

\begin{lem}
\label{lem:orthogonality}
The eigenspace decomposition in (M4) is orthogonal with respect to $(\, , \, )$ (i.e. $(u,v) = 0$ for all $u \in V_{\mu}^{(a)}$, $v \in V_{\nu}^{(a)}$ with $\mu, \nu \in \{ 1, 0, \frac{1}{2^2}, \frac{1}{2^5}\}$ and $\mu \neq \nu$).
\end{lem}

\begin{lem}
\label{lem:projection}
Let $V$ be a real vector space satisfying the axioms (M1) to (M5), and let $a$ be a Majorana axis of $V$. Then for all $v \in V$, the projection of $v$ onto the eigenspace $V_1^{(a)}$ is $(a,v)a$. 
\end{lem}

\begin{lem}
\label{lem:32evecs}
Let $V$ be a Majorana algebra and let $a \in V$ be a Majorana axis. If $v \in V$ then $v - v^{\tau(a)} \in V_{\frac{1}{2^5}}^{(a)}$.
\end{lem}

The following lemma is referred to in the literature as the \emph{fusion rules} of Majorana algebras. 

\begin{lem}
\label{lem:fusion}
For a fixed Majorana axis $a$, if $u \in V_{\mu}^{(a)}$, $v \in V_{\nu}^{(a)}$ then the product $u \cdot v$ lies in the sum of eigenspaces with corresponding eigenvalues given by the \mbox{$(\mu,\nu)$-th} entry of Table \ref{tab:fusion}.
\end{lem}

\begin{table}
\centering
\begin{tabular}{>{$}c<{$}|>{$}c<{$}>{$}c<{$}>{$}c<{$}>{$}c<{$}}
  & 1 & 0 & \frac{1}{2^2} & \frac{1}{2^5} \\ \hline
1 & 1 & 0 & \frac{1}{2^2} & \frac{1}{2^5} \\
0 & 0 & 0 & \frac{1}{2^2} & \frac{1}{2^5} \\
\frac{1}{2^2} & \frac{1}{2^2} & \frac{1}{2^2} & 1,0 & \frac{1}{2^5} \\
\frac{1}{2^5} & \frac{1}{2^5} & \frac{1}{2^5} & \frac{1}{2^5} & 1,0,\frac{1}{2^2} 
\end{tabular}
\caption{The Fusion Rules}
\label{tab:fusion}
\end{table}

The seminal paper in Majorana theory was that of Ivanov et. al. \cite{IPSS10}. In the first part of the paper, they reproved Sakuma's theorem (Theorem \ref{thm:IPSS10}) as stated below. This result truly forms the foundation of Majorana theory. It can equivalently be thought of as the classification of $2$-generated Majorana algebras, i.e. those generated by at most two Majorana axes.

\begin{thm}[{{\cite{IPSS10}}}] 
\label{thm:IPSS10}
Let $\textbf{R} = (G,T,V,\cdot, (,), \varphi, \psi)$ be a Majorana representation of $G$, as defined above. For $t_0, t_1 \in T$ let $a_0 = \psi(t_0)$, $a_1 = \psi(t_1)$, $\tau_0 = \varphi(t_0)$, $\tau_1 = \varphi(t_1)$ and $\rho = t_0t_1$. Finally, let $D \leq GL(V)$ be the dihedral group $\langle \tau_0, \tau_1 \rangle$. Then 
\begin{enumerate}[(i)]
\item $|D| = 2N$ for $1 \leq N \leq 6$;
\item the subalgebra $U = \langle \langle a_0, a_1 \rangle \rangle$ is isomorphic to a dihedral algebra of type $NX$ for $X \in \{A,B,C\}$, the structure of which is given in Table \ref{tab:sakuma};
\item for $i \in \mathbb{Z}$ and $\epsilon \in \{0,1\}$, the image of $a_{\epsilon}$ under the $i$-th power of $\rho$, which we denote $a_{2i+\epsilon}$, is a Majorana axis and $\tau(a_{2i + \epsilon}) = \rho^{-i}\tau_{\epsilon}\rho^i$. 
\end{enumerate}
\end{thm}

\begin{table}
\begin{center}
\vspace{0.35cm}
\noindent
\begin{tabular}{|c|c|c|}
\hline
&&\\
 Type & Basis & Products and angles \\
&&\\
\hline
&&\\
&& $a_0 \cdot a_1=\frac{1}{2^3}(a_0+a_1-a_{\rho(t_0,t_1)}),~a_0 \cdot a_{\rho(t_0,t_1)}=\frac{1}{2^3}(a_0+a_{\rho(t_0,t_1)}-a_1)$ \\
2A & $a_0,a_1,a_{\rho(t_0,t_1)}$ & $a_{\rho(t_0,t_1)} \cdot a_{\rho(t_0,t_1)} = a_{\rho(t_0,t_1)}$ \\ 
&&$(a_0,a_1)=(a_0,a_{\rho(t_0,t_1)})=(a_{\rho(t_0,t_1)},a_{\rho(t_0,t_1)})=\frac{1}{2^3}$\\
&& \\ 
2B & $a_0,a_1$ &$a_0 \cdot a_1=0$,~$(a_0,a_1)=0$ \\
&&\\
&  &$a_0 \cdot a_1=\frac{1}{2^5}(2a_0+2a_1+a_{-1})-\frac{3^3 \cdot 5}{2^{11}}u_{\rho(t_0,t_1)}$\\
3A& $a_{-1},a_0,a_1,$ & $a_0 \cdot u_{\rho(t_0,t_1)}=\frac{1}{3^2}(2a_0-a_1-a_{-1})+\frac{5}{2^5}u_{\rho(t_0,t_1)}$~~~~\\
&$u_{\rho(t_0,t_1)}$& $u_{\rho(t_0,t_1)} \cdot u_{\rho(t_0,t_1)}=u_{\rho(t_0,t_1)}$\\
&& $(a_0,a_1)=\frac{13}{2^8}$,~$(a_0,u_{\rho(t_0,t_1)})=\frac{1}{2^2}$,~$(u_{\rho(t_0,t_1)},u_{\rho(t_0,t_1)})=\frac{2^3}{5}$
\\
&&\\
3C & $a_{-1},a_0,a_1$ & $a_0 \cdot a_1=\frac{1}{2^6}(a_0+a_1-a_{-1}),~(a_0,a_1)=\frac{1}{2^6}$ \\
&&\\ 
&  & ~$a_0 \cdot a_1=\frac{1}{2^6}(3a_0+3a_1+a_2+a_{-1}-3v_{\rho(t_0,t_1)})$\\
4A & $a_{-1},a_0,a_1,$ & $a_0 \cdot v_{\rho(t_0,t_1)}=\frac{1}{2^4}(5a_0-2a_1-a_2-2a_{-1}+3v_{\rho(t_0,t_1)})$\\
&$a_2,v_{\rho(t_0,t_1)}$&~$v_{\rho(t_0,t_1)} \cdot v_{\rho(t_0,t_1)}=v_{\rho(t_0,t_1)}$, ~$a_0 \cdot a_2=0$ \\
& & $(a_0,a_1)=\frac{1}{2^5},~(a_0,a_2)=0,~(a_0,v_{\rho(t_0,t_1)})=\frac{3}{2^3},~(v_{\rho(t_0,t_1)},v_{\rho(t_0,t_2)})=2$\\
&&\\
4B & $a_{-1},a_0,a_1,$ & $a_0 \cdot a_1=\frac{1}{2^6}(a_0+a_1-a_{-1}-a_2+a_{\rho(t_0,t_2)})$
\\
& $a_2,a_{\rho(t_0,t_2)}$ & $a_0 \cdot a_2=\frac{1}{2^3}(a_0+a_2-a_{\rho(t_0,t_2)})$\\
&& $(a_0,a_1)=\frac{1}{2^6},~(a_0,a_2)=(a_0,a_{\rho(t_0,t_1)})=\frac{1}{2^3}$ \\
&&\\
&& $a_0 \cdot a_1=\frac{1}{2^7}(3a_0+3a_1-a_2-a_{-1}-a_{-2})+w_{\rho(t_0,t_1)}$
\\
 5A & $a_{-2},a_{-1},a_0,$ & $a_0 \cdot a_2=\frac{1}{2^7}(3a_0+3a_2-a_1-a_{-1}-a_{-2})-w_{\rho(t_0,t_1)}$
\\
& $a_1,a_2,w_{\rho(t_0,t_1)}$ & $a_0 \cdot w_{\rho(t_0,t_1)}=\frac{7}{2^{12}}(a_{1}+a_{-1}-a_2-a_{-2})+\frac{7}{2^5}w_{\rho(t_0,t_1)}$\\
& & $w_{\rho(t_0,t_1)} \cdot w_{\rho(t_0,t_1)}=\frac{5^2 \cdot 7}{2^{19}}(a_{-2}+a_{-1}+a_0+a_1+a_2)$\\
&&$(a_0,a_1)=\frac{3}{2^7},~(a_0,w_{\rho(t_0,t_1)})=0$, $(w_{\rho(t_0,t_1)},w_{\rho(t_0,t_1)})=\frac{5^3 \cdot 7}{2^{19}}$\\
&& \\
& & $a_0 \cdot a_1=\frac{1}{2^6}(a_0+a_1-a_{-2}-a_{-1}-a_2-a_3+a_{\rho(t_0,t_3)})+\frac{3^2 \cdot 5}{2^{11}}u_{\rho(t_0,t_2)}$\\
6A& $a_{-2},a_{-1},a_0,$ &$a_0 \cdot a_2=\frac{1}{2^5}(2a_0+2a_2+a_{-2})-\frac{3^3 \cdot 5}{2^{11}}u_{\rho(t_0,t_2)}$  \\ 
&$a_1,a_2,a_3$  &$a_0 \cdot u_{\rho(t_0,t_2)}=\frac{1}{3^2}(2a_0-a_2-a_{-2})+\frac{5}{2^5}u_{\rho(t_0,t_2)}$  \\
&$a_{\rho(t_0,t_3)},u_{\rho(t_0,t_2)}$ & $a_0 \cdot a_3=\frac{1}{2^3}(a_0+a_3-a_{\rho(t_0,t_3)})$, $a_{\rho(t_0,t_3)} \cdot u_{\rho(t_0,t_2)}=0$\\
&&$(a_{\rho(t_0,t_3)},u_{\rho(t_0,t_2)})=0$, $(a_0,a_1)=\frac{5}{2^8}$, $(a_0,a_2)=\frac{13}{2^8}$, $(a_0,a_3)=\frac{1}{2^3}$\\
&&\\
\hline
\end{tabular}
\caption{The dihedral Majorana algebras}
\label{tab:sakuma}  
\end{center}
\end{table}

We can use the values in Table \ref{tab:sakuma} to calculate the eigenvectors of these algebras with respect to the axis $a_0$. 

\begin{prop}
The eigenspace decompositions with respect to the axis $a_0$ for each dihedral Majorana algebra is given in Table \ref{tab:evecs}. In each case, the $1$-eigenspace is the $1$-dimensional space spanned by $a_0$ and so is omitted from this table. 
\end{prop}

\begin{table}
\begin{center}
\vspace{0.35cm}
\noindent
\begin{tabular}{|>{$}c<{$}|>{$}c<{$}|>{$}c<{$}|>{$}c<{$}|}
\hline
&&&\\
 \mbox{Type} & 0 & \frac{1}{2^2} & \frac{1}{2^5}  \\
&&&\\
\hline
&&&\\
2A & a_1 + a_{\rho(t_0,t_1)} - \frac{1}{2^2} & a_1 - a_{\rho(t_0,t_1)} & \\
&&&\\
2B & a_1 & & \\
&&&\\
3A  & u_{\rho(t_0,t_1)} + \frac{2 \cdot 5}{3^3}a_0 + \frac{2^5}{3^3}(a_1 + a_{-1}) 
    & \ml{c}{u_{\rho(t_0,t_1)} - \frac{2^3}{3^2 \cdot 5}a_0\\ 
        - \frac{2^5}{3^2 \cdot 5}(a_1 + a_{-1}) }
    & a_1 - a_{-1} \\
&&&\\
3C  & a_1 + a_{-1} - \frac{1}{2^5}a_0 
    & 
    & a_1 - a_{-1} \\ 
&&&\\
4A  & v_{\rho(t_0,t_1)} - \frac{1}{2}a_0 + 2(a_1 + a_{-1}), a_2 
    & \ml{c}{v_{\rho(t_0,t_1)} - \frac{1}{3}a_0 \\ 
        - \frac{2}{3}(a_1 + a_{-1}) - \frac{1}{3}a_2} 
    & a_1 - a_{-1} \\ 
&&&\\
4B  & \ml{c}{a_1 + a_{-1} - \frac{1}{2^5}a_0 - \frac{1}{2^3}(a_{\rho(t_0,t_2)}- a_2), \\ 
        a_2 + a_{\rho(t_0,t_2)} - \frac{1}{2^2}a_0} 
    & a_2 - a_{\rho(t_0,t_2)} 
    & a_1 - a_{-1} \\ 
&&&\\
5A  & \ml{c}{w_{\rho(t_0,t_1)} + \frac{3}{2^9}a_0 - \frac{3 \cdot 5}{2^7}(a_1 + a_{-1}) - \frac{1}{2^7}(a_2 - a_{-2}), \\ 
        w_{\rho(t_0,t_1)} - \frac{3}{2^9}a_0 + \frac{1}{2^7}(a_1 + a_{-1}) + \frac{3 \cdot 5}{2^7}(a_2 + a_{-2})} 
    & \ml{c}{w_{\rho(t_0,t_1)} + \frac{1}{2^7}(a_1 + a_{-1}) \\
        - \frac{1}{2^7}(a_2 + a_{-2})} 
    & \ml{c}{a_1 - a_{-1}, \\ a_2 - a_{-2}} \\
&&&\\
6A  & \ml{c}{u_{\rho(t_0,t_2)} + \frac{2}{3^2 \cdot 5}a_0 - \frac{2^8}{3^2 \cdot 5}(a_1 - a_{-1}) \\ 
        - \frac{2^5}{3^2 \cdot 5}(a_2 + a_{-2} + a_3 - a_{\rho(t_0,t_3)}), \\ 
        a_3 + a_{\rho(t_0,t_3)} - \frac{1}{2^2}a_0, \\ 
        u_{\rho(t_0,t_2)} - \frac{2 \cdot 5}{3^3}a_0 + \frac{2^5}{3^3}(a_2 + a_{-2})} 
    & \ml{c}{u_{\rho(t_0,t_2)} - \frac{2^3}{3^2 \cdot 5}a_0  \\ 
        - \frac{2^5}{3^2 \cdot 5}(a_2 + a_{-2} + a_3)\\
        + \frac{2^5}{3^2 \cdot 5} a_{\rho(t_0,t_3)} , \\ 
        a_3 - a_{\rho(t_0,t_3)}} 
    & \ml{c}{a_1 - a_{-1}, \\ a_2 - a_{-2}} \\
&&& \\
\hline
\end{tabular}
\caption{The eigenspace decomposition of the dihedral Majorana algebras}
\label{tab:evecs}  
\end{center}
\end{table}

The following result is also a direct consequence of the values in Table \ref{tab:sakuma}. 

\begin{prop}
\label{prop:axes}
Let $V$ be a Majorana algebra and let $a_0, a_1 \in A$. Let $U := \langle \langle a_0, a_1 \rangle \rangle$ and let $t_0 := \tau(a_0)$ and $t_1 := \tau(a_1)$. If $U$ is of type $3A$ or $4A$ and then $U$ contains the additional basis vector $u_{\rho(t_0, t_1)}$ or $v_{\rho(t_0,t_1)}$ respectively where $u_{\rho(t_0,t_1)}$ and $v_{\rho(t_0,t_1)}$ depend only on the cyclic subgroup $\langle t_0t_1 \rangle$. That is to say,
\[
u_{\rho(t_0,t_1)} = u_{\rho(t_0, t_0t_1t_0)} \mbox{ and } v_{\rho(t_0, t_1)} = v_{\rho(t_0, t_0t_1t_0)}.
\]
Similarly, if $U$ is of type $5A$ then $U$ contains an additional basis vector $w_{\rho(t_0,t_1)}$ which, up to a possible change of sign, depends only on the cyclic subgroup $\langle t_0t_1 \rangle$, as below
\begin{equation}
\label{eq:5Aaxes}
w_{\rho(t_0,t_1)} = - w_{\rho(t_0, t_1t_0t_1)} = - w_{\rho(t_0, t_0t_1t_0t_1t_0)} = w_{\rho(t_0, t_0t_1t_0)}.
\end{equation}
\end{prop}

Sakuma's theorem provides a crucial tool in Majorana theory. If the dihedral subalgebras of a Majorana algebra are known, they provide initial values for the inner and algebra products which can be used to explore the structure of the whole algebra. As such, it is important to classify the possibilities for the dihedral algebras which are contained in a given Majorana algebra. This idea is formalised in the following definition. 

\begin{defn}
If $\textbf{R} = (G,T,V)$ is a Majorana representation then we define a map $\Psi$ which sends $(t,s) \in T^2$ to the type of the dihedral Majorana algebra $\langle \langle a_{t}, a_{s}  \rangle \rangle$. Then the \emph{shape} of $\textbf{R}$ is the tuple $[\Psi((t_{i_1}, t_{j_1})), \Psi((t_{i_2}, t_{j_2}))  \ldots , \Psi((t_{i_n}, t_{j_n}))]$ where the $(t_{i_k}, t_{j_k})$ are representatives of the orbits of $G$ on $T \times T$. 
\end{defn}

It is worth noting that it is not necessarily the case that two non-isomorphic Majorana representations must have different shapes. However, there are no known examples where this is not the case. It is an open question as to whether there exist any non-isomorphic Majorana algebras which share the same shape. 

When determining the shape of a Majorana representation, if $t, s \in T$ and $\langle t,s \rangle \cong D_{2N}$ then $\Psi((t,s)^G) = NX$ for $X \in \{A, B, C \}$. If $N = 1, 5$ or $6$, there is only one option for the value of $X$. If $N$ takes any other value, we must use other results to restrict the possible value of $X$. In particular, the following lemma can be deduced from the structure of the dihedral algebras (see Lemma 2.20, \cite{IPSS10}).
\begin{lem}
\label{lem:inclusions}
Let $U$ be a dihedral algebra (as in Table \ref{tab:sakuma}) that is generated by Majorana axes $a_0$ and $a_1$. Then
\begin{enumerate}[(i)]
\item if $U$ is of type $4A$, $4B$ or $6A$ then the subalgebra generated by $a_0$ and $a_2$ is of type $2B$, $2A$ or $3A$ respectively;
\item if $U$ is of type $6A$ then the subalgebra generated by $a_0$ and $a_3$ is of type $2A$. 
\end{enumerate} 
\end{lem}

Informally, this means that we have the following \emph{inclusions} of algebras:
\[
2A \hookrightarrow 4B, \qquad 2B \hookrightarrow 4A, \qquad 2A \hookrightarrow 6A, \qquad \textrm{and} \qquad 3A \hookrightarrow 6A.
\]

Given a group $G$ and a set of $G$-invariant involutions $T$, we let $\Gamma$ be the directed graph whose vertex set is the set of orbits of $G$ on $T \times T$ and where $(t_0, t_1)^G \rightarrow (t_2, t_3)^G$ is an edge if and only if $\langle \langle \psi(t_0), \psi(t_1) \rangle \rangle \hookrightarrow \langle \langle \psi(t_2), \psi(t_3) \rangle \rangle$. If we fix the type of the dihedral algebra corresponding to one vertex $v \in V(\Gamma)$ then this determines the types of the algebras corresponding to all vertices in its connected component. In particular, there are at most $2^c$ possible shapes for a representation of the form $(G,T,V)$, where $c$ is the number of connected components of $\Gamma$.

Finally, we define five further axioms. Let $(G,T,V)$ be a Majorana representation and let $t_0, t_1, t_2, t_3 \in T$ with corresponding Majorana axes $a_i := \psi(t_i)$.
\begin{description}
\item[2Aa] If $t_0t_1t_2 = 1$ and $\langle \langle a_0, a_1 \rangle \rangle$ is of type $2A$ then $a_2 \in \langle \langle a_0, a_1 \rangle \rangle$ and $a_2 = a_{\rho(t_0,t_1)}$.
\item[2Ab] If $\langle \langle a_0, a_1 \rangle \rangle$ and $\langle \langle a_2, a_3 \rangle \rangle$ are of type $2A$ and $t_0t_1 = t_2t_3 $ then the basis elements $a_{\rho(a_0,a_1)}$ and $a_{\rho(a_2, a_3)}$ are equal.
\item[3A] If $\langle \langle a_0, a_1 \rangle \rangle$ and $\langle \langle a_2, a_3 \rangle \rangle$ are of type $3A$ and $\langle t_0t_1 \rangle = \langle t_2t_3 \rangle$ then the basis elements $u_{\rho(a_0,a_1)}$ and $u_{\rho(a_2, a_3)}$ are equal.
\item[4A] If $\langle \langle a_0, a_1 \rangle \rangle$ and $\langle \langle a_2, a_3 \rangle \rangle$ are of type $4A$ and $\langle t_0t_1 \rangle = \langle t_2t_3 \rangle$ then the basis elements $v_{\rho(a_0,a_1)}$ and $v_{\rho(a_2, a_3)}$ are equal.
\item[5A] If $\langle \langle a_0, a_1 \rangle \rangle$ and $\langle \langle a_2, a_3 \rangle \rangle$ are of type $4A$ and $\langle t_0t_1 \rangle = \langle t_2t_3 \rangle$ then the basis elements $w_{\rho(a_0,a_1)}$ and $w_{\rho(a_2, a_3)}$ are equal, up to a possible change of sign.
\end{description}

These axioms, as well as the axioms M1 - M7, are known to hold in the Griess algebra. We have implemented two versions of the algorithm, one which assumes axioms 2Aa - 5A and one which does not. However, it is usually much more efficient to assume these axioms as they limit the size of the spanning set \texttt{coordinates} as described in Section \ref{sec:implementation}. In particular, as the main aim of this work is the construction of large subalgebras of the Griess algebra, we assume these axioms throughout the rest of this paper. 

Henceforth, if $U = \langle \langle a_0, a_1 \rangle \rangle$ is a dihedral algebra of type $3A$, $4A$ or $5A$ then we let $t_0 := \tau(a_0)$, $t_1 := \tau(a_1)$ and $h := t_0t_1$ and write $u_h := u_{\rho(t_0, t_1)}$, $v_h := v_{\rho(t_0, t_1)}$ or $w_h := w_{\rho(t_0, t_1)}$ respectively.

\section{Notes on the implementation}
\label{sec:implementation}

The construction of these algebras is expensive both in terms of time and memory and a large part of the implementation of the algorithm involves mitigating these factors. In particular, we exploit the fact that the algebra and inner products on the algebra are preserved by the action of the group on the algebra. That is to say, for all $u, v \in V$ and $g \in G$, $(u^g, v^g) = (u,v)$ and $u^g \cdot v^g = (u \cdot v)^g$. 

In particular, we store product values only for representatives of the orbits of $G$ on $C \times C$ where $C$ is a spanning set of $V$. In order to find the algebra product of two generic vectors, we must recover the required values from the corresponding orbit representative using the action of the group. Similarly, we store eigenvectors only for representatives of the orbits of $G$ on the axes of $V$. 

Recall from Proposition \ref{prop:axes} that we have the following equalities on $3A$, $4A$ and $5A$ axes:
\[
u_h = u_{h^2}, \, v_{h} = v_{h^3}, \mbox{ and } w_h = -w_{h^2} = -w_{h^3} = w_{h^4}. 
\]
In the implementation of the algorithm, we exploit these equalities, as well as those from axioms 2Aa - 5A above, whilst keeping track of any sign changes from the $5A$ axes. 

In a departure from Seress' methods, given a spanning set $C$ of $V$ of size $n$, we express the elements of $G$ as signed permutations on the $n$ vectors of the set $C$. This is more efficient both in terms of time and memory and also makes it easier to implement an $n$-closed version of the algorithm (see Section \ref{sec:nclosed}). In particular, we express an element $g$ of $G$ as
\[
[\pm i_1, \pm i_2, \ldots, \pm i_n]
\]
where $C[j]^g = \pm C[i_j]$. 

For each algebra, we store the following data structures which enable the calculation of the signed permutation corresponding to a given element. 
\begin{itemize}
\item \texttt{coordinates:} This is a set (a sorted, duplicate-free list) which we denote $C$ consisting of all elements of $T$, as well as, or including, exactly one generator of each cyclic group of the form $\langle t_0t_1 \rangle$ for $t_0, t_1 \in T$ such that the dihedral algebra $\langle \langle \psi(t_0), \psi(t_1) \rangle \rangle$ is of type $2A$ $3A$, $4A$ or $5A$. 
\item \texttt{longcoordinates:} This is a set consisting of all elements of $T$ as well as, or including, \emph{all} generators of each cyclic group $\langle t_0t_1 \rangle$ for $t_0, t_1 \in T$ such that the dihedral algebra $\langle \langle \psi(t_0), \psi(t_1) \rangle \rangle$ is of type $2A$ $3A$, $4A$ or $5A$.
\item \texttt{positionlist:} This is a list whose order is equal to the cardinality of \texttt{longcoordinates}. The absolute value of \texttt{positionlist[i]} is the index of the element of \texttt{coordinates} which corresponds to \texttt{longcoordinates[i]}. The entry \texttt{positionlist[i]} is negative if and only if \texttt{longcoordinates[i]} is of order $5$ and is equal to $h^2$ or $h^3$, where $h$ is the corresponding element of \texttt{coordinates}.
\end{itemize}

We note that in the case of the $n$-closed algorithm, we store further elements in the lists \texttt{coordinates} and \texttt{longcoordinates} as described in detail in Section \ref{sec:nclosed}.

We now describe how we use these signed permutations to recover generic product values. We use a bespoke orbits algorithm which partitions the set \texttt{longcoordinates} $\times$ \texttt{longcoordinates} according to the equivalence relation $\sim$ defined as below. 
\begin{defn}
For all $u, v, w, z \in V$, $(u,v) \sim (w,z)$ if and only if at least one of the following conditions holds:
\begin{itemize}
\item $(u,v) = (w^g, z^g)$ for some $g \in G$;
\item $(u,v) = (z,w)$;
\item $\langle u \rangle = \langle w \rangle$ and $\langle v \rangle = \langle z \rangle$.
\end{itemize}
\end{defn}

We label these equivalence classes $P_1, \ldots, P_k$. We do not need to explicitly store these classes. Instead, our bespoke orbits algorithm outputs the following structures.
\begin{itemize}
\item \texttt{pairrepresentatives:} From each equivalence class $P_i$, we pick a representative $p_i$. The entry \texttt{pairrepresentatives[i]} is a list of length two consisting of the positions in \texttt{coordinates} of the elements of $p_i$.
\item \texttt{pairorbitlist:} This is a matrix of size $|C| \times |C|$ whose entries lie in $\{1,2, \ldots, k\}$. If $[\texttt{coordinates[i]}, \texttt{coordinates[j]}] \in P_k$ then \texttt{pairorbitlist[i][j]} is equal to $k$. 
\item \texttt{pairconjelements:} This is a list of length $|G|$ consisting of the signed permutations corresponding to each element of $G$. 
\item \texttt{pairconj:} This is a matrix of size $|C| \times |C|$ whose entries lie in $\{1,2, \ldots, |G|\}$. If [\texttt{coordinates[i]}, \texttt{coordinates[j]}] $\in P_k$ then the value of \texttt{pairconj[i][j]} gives the index in \texttt{pairconjelements} of a signed permutation which sends the representative $p_k$ of $P_k$ to [\texttt{coordinates[i]}, \texttt{coordinates[j]}].
\end{itemize}

In order to recover the product of the two basis vectors corresponding to the elements $C[i]$ and $C[j]$, we execute the following steps.
\begin{enumerate}
\item Let \texttt{k := pairorbit[i][j]} and let \texttt{l := pairconj[i][j]}.
\item Then \texttt{u := algebraproducts[k]} will be a row vector and \texttt{g := pairconjelements[l]} will be a signed permutation representing a group element which sends the representative of the orbit $P_k$ to \texttt{[coordinates[i], coordinates[j]]}. 
\item The desired product will be equal to the row vector \texttt{v} where
\begin{align*}
\texttt{v[i] :=} 
\begin{cases}
\texttt{u[g[i]]} &\mbox{ if } \texttt{g[i]} > 0 \\
\texttt{-u[-g[i]]} &\mbox{ if } \texttt{g[i]} < 0. 
\end{cases}
\end{align*} 
\end{enumerate}

Finally, we store eigenvectors only for representatives of the orbits of $G$ on $T$. Again, instead of storing the full orbits, we use a bespoke orbits algorithm which outputs the following two structures.
\begin{itemize}
\item \texttt{orbitrepresentatives:} This is a list whose entries give the indices in \texttt{coordinates} of representatives of orbits of $G$ on $T$.
\item \texttt{conjelements:} This is a duplicate free list of signed permutations corresponding to the group elements which send an element of \texttt{coordinates} to one of the representatives in \texttt{orbitrepresentatives}.
\end{itemize}

Finally, we note that the methods used in this algorithm require the solving of potentially large systems of linear equations over the rational numbers. Storing the matrices involved and reducing them to row echelon form takes a large amount of the memory and time required by the program. In what we believe to be an improvement on Seress' methods, we use the sparse matrix format provided by the GAP package \texttt{Gauss} \cite{GaussPkg17} as an efficient way to store and compute with the matrices in question.

\section{The algorithm}
\label{sec:mainloop}

\paragraph{Input:} A finite group $G$ and a $G$-invariant set $T$ of involutions such that $G = \langle T \rangle$. The user must also choose which of the possible shapes found by the function \texttt{ShapesOfMajoranaRepresentation} (see Step 1) is to be considered by the algorithm. 

\paragraph{Output:} The algorithm returns a record with the following components.
\begin{itemize}
\item \texttt{group} and \texttt{involutions:} The group $G$ and generating set of involutions $T$, as inputted by the user.
\item \texttt{shape:} The shape of the representation, as chosen by the user.
\item \texttt{setup:} A record whose components are the 9 structures given in Section \ref{sec:implementation}.
\item \texttt{algebraproducts:} A list of row vectors (in sparse matrix format) where \texttt{algebraproducts[i]} gives the algebra product of the two basis vectors whose indices are given by \texttt{setup.pairreps[i]}.
\item \texttt{innerproducts:} A list where \texttt{innerproducts[i]} gives the inner product of the two basis vectors whose indices are given by \texttt{setup.pairreps[i]}.
\item \texttt{evecs:} If $i$ is in \texttt{setup.orbitrepresentatives} then for $j = 1,2,3$, \texttt{evecs[i][j]} gives respectively a basis (in sparse matrix format) of the $0$-, $\frac{1}{2^2}$- or $\frac{1}{2^5}$-eigenspace of the $i$th axis.
\item \texttt{nullspace:} A matrix (in sparse matrix format) which forms a basis of the nullspace of the algebra. 
\end{itemize}

\subsection*{Step 1 - Shapes}

The first step is to find all possible shapes of a Majorana representation of the form $(G,T,V)$. That is to say, we find representatives of the orbits of $G$ on $T \times T$ and determine the possibilities for the types of the dihedral algebras generated by the Majorana axes corresponding to each of these representatives. Note that the possible shapes must respect the inclusions of dihedral algebras, as described in Lemma \ref{lem:inclusions}. 

The function \texttt{ShapesOfMajoranaRepresentation} takes as its input a group $G$ and a set of involutions $T$ and returns a record, of which one of the components is labelled \texttt{shapes} and gives a list of possible shapes for a representation of the form $(G,T,V)$. The user may then choose which of these possible shapes they want to use for the constructive part of the algorithm.

This output is then used as the first input variable in the function \texttt{MajoranaRepresentation}. The second variable is an integer $i$ which signifies that the user has chosen the shape at the position $i$. 

\subsection*{Step 2 - Set Up} 

The first step is to build the nine objects which form the record \texttt{setup}, as described in Section \ref{sec:implementation}. We then record all product values and eigenvectors which are given from the known values on dihedral algebras, i.e. those in Tables \ref{tab:sakuma} and \ref{tab:evecs}.

\subsection*{Step 3 - Inner Products}

The first step in the main part of the algorithm is to find inner product values on the spanning set $C$ of $V$. Let $u$, $v$ and $w$ be elements of $C$. Then, from axiom M1
\[
(u, v \cdot w) = (u \cdot v, w).
\]
We consider all cases where the algebra products $v \cdot w$ and $u \cdot v$ are known, but at least one of the inner products above are not. Using equations of this form, we build systems of linear equations of the unknown products $(u, v)$ where $u$ and $v$ are elements of $C$. We then solve this system and record any new inner products. 

If at this stage, all inner product values are known then we construct the Gram matrix of the inner product and store basis vectors of its nullspace as the component \texttt{nullspace}. In particular, as the inner product is positive definite, these vectors will form a basis of the nullspace of the algebra. 

We note that Seress instead uses the orthogonality of eigenvectors (Lemma \ref{lem:orthogonality}) to determine new inner products. We have tested both methods and have found that our approach tends to find more products and is more efficient. However, in either case, finding inner products tends to take only a small amount of the total running time. 

\subsection*{Step 4 - Fusion of Eigenvectors}

For each of the orbit representatives of $G$ on $T$ given by \texttt{setup.orbitrepresentatives} we take the corresponding Majorana axis $a$ and consider all pairs of known eigenvectors $\alpha \in V_{\mu}^{(a)}$ and $\beta \in V_{\nu}^{(a)}$ for $\mu, \nu \in \{0, \frac{1}{2^2}, \frac{1}{2^5}\}$. If the product $\alpha \cdot \beta$ is known then we can use this, along with the fusion rules, to find further eigenvectors. 

If $\mu \neq \nu$, or if $\mu = \nu = 0$ then $\alpha \cdot \beta$ is itself an eigenvector, and is added to the relevant eigenspace. If $\mu = \nu = \frac{1}{2^2}$ then, using the fusion rules and Lemma \ref{lem:projection}. 
\[
\alpha \cdot \beta = v_0 + (a, \alpha \cdot \beta) a = v_0 + \frac{1}{2^2}(\alpha, \beta)a
\]
where $v_0 \in V_{0}^{(a)}$. Thus, if the value of $(\alpha, \beta)$ is known, then we can recover the eigenvector $v_0$ and add it to the $0$-eigenspace of $a$. 

Similarly, if $\mu = \nu = \frac{1}{2^5}$ then 
\[
\alpha \cdot \beta = v_0 + v_{\frac{1}{2^2}} + \frac{1}{2^5}(\alpha, \beta)a
\]
where $v_0 \in V_{0}^{(a)}$ and $v_{\frac{1}{2^2}} \in V_{\frac{1}{2^2}}^{(a)}$ If the value of $(\alpha, \beta)$ is known, then we further calculate
\[
a \cdot (\alpha \cdot \beta) =  \frac{1}{2^2}v_{\frac{1}{2^2}} + \frac{1}{2^5}(\alpha, \beta)a
\]
and so we can recover both $v_0$ and $v_{\frac{1}{2^2}}$ and add them to their respective eigenspaces.

\subsection*{Step 5 - Algebra Products }

We seek products of the form $u \cdot v$ for $u,v \in C$. We write a system of linear equations of the unknown products from the following sources:
\begin{itemize}
	\item if $v \in V_{\mu}^{(a)}$ for some $a \in A$ and $\mu \in \{0,\frac{1}{2^2},\frac{1}{2^5}\}$ then $a \cdot v = \mu v$; 
	\item if $u \in C$ and $v \in N(V)$ then $u \cdot v = 0$. 
\end{itemize}

We then solve this system and record any new algebra products. If there remains unknown algebra products then we again construct a system of linear equations, this time making use of the resurrection principle.

\begin{prop}[The Generalised Resurrection Principle]
\label{prop:genresurrection}
Fix $a \in A$ and let $\alpha, \gamma \in V_{0}^{(a)}$ and $\beta \in V_{\mu}^{(a)}$ for $\mu \in \{\frac{1}{2^2},\frac{1}{2^5} \}$. Then 
\begin{equation*}
\label{eq:genresurrection}
a \cdot ((\alpha - \beta) \cdot \gamma) = - \mu (\gamma \cdot \beta).
\end{equation*}
Similarly, if $\alpha, \gamma \in V_{\frac{1}{2^2}}^{(a)}$ and $\beta \in V_{\mu}^{(a)}$ for $\mu \in \{0,\frac{1}{2^5} \}$. Then 
\begin{equation*}
a \cdot ((\alpha - \beta) \cdot \gamma) = \frac{1}{2^2}(\alpha, \gamma) a - \nu (\gamma \cdot \beta)
\end{equation*}
where $\nu = \frac{1}{2^2}$ if $\mu = 0$ and $\nu = \frac{1}{2^5}$ if $\mu = \frac{1}{2^5}$.
\end{prop}

\begin{proof}
Firstly,
\begin{align*}
a \cdot ((\alpha - \beta) \cdot \gamma) = a \cdot (\alpha \cdot \gamma) - a \cdot ( \beta \cdot \gamma).
\end{align*}
The result then follows from the fusion rules.  
\end{proof}

In particular, where possible, we choose $\alpha$, $\beta$ and $\gamma$ such that the product $(\alpha - \beta) \cdot \gamma$ is known, but $\beta \cdot \gamma$ is not known. In this way, we obtain a linear combination of terms $a \cdot x$ for the $x \in C$ occuring in $(\alpha - \beta) \cdot \gamma $, and $y \cdot z$ for $y \in \beta$ and $z \in \gamma$. Using these, we construct and solve a system of linear equations.

In some cases, there still remains unknown algebra and inner product values and so we run steps 3 - 5 repeatedly until all values have been found. In some rare cases, it is not possible to find all products and the program exits with the output \texttt{fail}. Where the algorithm has sucessfully completed and the representation in question is not known to embed into the Griess algebra then we run checks to ensure that the constructed algebra is indeed a Majorana algebra.

\section{The $n$-closed algorithm}
\label{sec:nclosed}

In a significant improvement on Seress' work, we have been able to implement a version of the algorithm which allows the construction of $n$-closed representations, theoretically for any given $n$. Of course, technology limits the value of $n$ in practice and, in fact, we do not know of any algebras which are $n$-closed for $n > 3$ which are not also $3$-closed with respect to some set of Majorana axes. We have implemented this part of the algorithm so that the first step is to attempt construction of the $2$-closed algebras as, in the vast majority of cases, this is sufficient. 

If the algorithm is unable to determine all algebra products from the $2$-closed part of $V$ then the user may pass the incomplete algebra outputted by the function \texttt{MajoranaRepresentation} to the function \texttt{NClosedMajoranaRepresentation} in order to attempt construction of the $3$-closed algorithm. In order to attempt construction of the $n$-closed part of the algebra, the user must pass the incomplete algebra to the function \texttt{NClosedMajoranaRepresentation} $n - 2$ times for $n > 2$. 

We describe the implementation of the function \texttt{NClosedMajoranaRepresentation}. We note that this method crucially relies on our use of signed permutations to encode the action of the group on the algebra, as described in Section \ref{sec:implementation}.

The main steps of the $n$-closed algorithm are the same as those for the $2$-closed algorithm. The function \texttt{NClosedMajoranaRepresentation} extends the spanning set of the algebra and adjusts the record encoding the algebra as described below. We then perform Steps 3 - 5 of the main algorithm until no more algebra products can be found. We describe this process in detail below. 

\paragraph{Input:} A record which has been outputted by the function \texttt{MajoranaRepresentation} where at least one of the entries in the component \texttt{algebraproducts} has the value \texttt{false}, indicating a product which has not yet been determined. 
\paragraph{Output:} The function has no output. We record additional values in the components of the input record.

We first record a list of indices \texttt{k} such that \texttt{algebraproducts[k]} has the value \texttt{false}. For each entry \texttt{k} of this list, if the value of \texttt{algebraproducts[k]} is still \texttt{false}, then we perform the following steps. 

\subsubsection*{Step 1 - Extend the spanning set of the algebra}
For all $i, j \in \{1, \ldots, |C|\}$, if the absolute value of \texttt{pairorbitlist[i][j]} is equal to \texttt{k}, then we add the ordered pair \texttt{[i,j]} to the list \texttt{coordinates}. This corresponds to adding the vector $C[i] \cdot C[j]$ to the spanning set of the algorithm.

\subsubsection*{Step 2 - Extend the known values} For all known products in \texttt{algebraproducts} and all vectors in \texttt{evecs} and \texttt{nullspace}, extend the length of the vectors to the new cardinality of \texttt{coordinates} by adding zeroes to the end. Set the value of the entry \texttt{algebraproducts[k]} to be the vector with a one at position \texttt{i} and zeroes elsewhere, where \texttt{i} is such that \texttt{coordinates[i]} is equal to \texttt{pairrepresentatives[k]}.

\subsubsection*{Step 3 - Extend the signed permutations} Let \texttt{p} be an element of \texttt{pairconjelements} or \texttt{conjelements} corresponding to an element $g \in G$. Let \texttt{l} be such that \texttt{coordinates[l]} is equal to \texttt{[i,j]} where \texttt{[i,j]} is one of the additional elements recorded in Step 1. 

Let \texttt{x} be the ordered list consisting of the abolute values of \texttt{p[i]} and \texttt{p[j]}. Then
\[
\texttt{p[l] = } \pm \texttt{Position(coordinates, x)}
\]  
where \texttt{p[l] > 0} if and only if \texttt{p[i]*p[j] > 0}. 

The generality of the method of signed permutations means that this method works regardless of whether the vectors corresponding to \texttt{coordinates[i]} and \texttt{coordinates[j]} are in the $2$-closed part of the algebra and are thus stored as group elements, or whether they are in the $n$-closed part for $n > 2$ and are stored as ordered pairs of integers.

\subsubsection*{Step 4 - New orbits} We can now use the orbits function developed for the main algorithm along with the signed permutations found in the previous step to find any additional orbits of $G$ on \texttt{coordinates} $\times$ \texttt{coordinates}. We add any new entries to \texttt{pairrepresentatives} and add new values to the matrices \texttt{pairorbitlist} and \texttt{pairconj}.

\subsubsection*{Step 5 - New eigenvectors} Finally, if $t \in T$ and $u, v \in V$ then, from Lemma \ref{lem:32evecs}, the vector $u \cdot v - (u \cdot v)^t$ is a $\frac{1}{2^5}$-eigenvector of the $2A$ axis $\psi(t)$. We use the signed permutations found in Step 3 to find and record any vectors of this form. 

\subsubsection*{Step 6 - The main steps} The extended algebra can now be passed through Steps 3 - 5 of Section \ref{sec:mainloop}. If at any point all products have been found then the algebra is complete and the function exits. If no more products can be found, but there are still missing values then we repeat the above Steps 1 - 6 with the next index \texttt{k} from our original list of unknown algebra product values.

If we have performed these steps for all values in this list and there still remains unknown products then the function exits and the algebra is still incomplete. The user may then decide to again run the function \texttt{NClosedMajoranaRepresentation} on the incomplete representation in order to attempt the construction of the $n$-closed algebra for the next value of $n$.

\section{Results}
\label{sec:results}

In Table \ref{tab:results} we give the basic details for some of the representations which we have constructed using our program. For each representation $(G, T, V)$, we give the following information
\begin{itemize}
\item the isomorphism type of $G$;
\item the cardinality of $T$, where $c_1 + c_2 + \ldots + c_k$ indicates that $T$ is the union of $k$ conjugacy classes of size $c_1, c_2, \ldots, c_k$;
\item a subset of the shape of $V$, showing only the values of $\Psi$ for the orbits where a choice has been made on the type of the corresponding dihedral algebra;
\item the cardinality of the spanning set $C$ of the $2$-closed part of the algebra, which consists of the $2A$, $3A$, $4A$ and $5A$ axes;
\item the dimension of the algebra $V$ as a vector space over $\mathbb{R}$;
\item whether the algebra is $2$-closed or not, all those which are not $2$-closed are $3$-closed. 
\end{itemize} 

\begin{table}
\begin{center}
\vspace{0.35cm}
\noindent
\begin{tabular}{|>{$}c<{$}|>{$}c<{$}>{$}c<{$}>{$}c<{$}>{$}c<{$}>{$}c<{$}>{$}c<{$}|} \hline 
&&&&&& \\
i & G & |T| & \textrm{Shape} & |C| & \textrm{dim.} & 2\textrm{-closed} \\ 
&&&&&& \\ \hline
&&&&&& \\
1 & S_4 & 6 + 3 & (2B,3C) & 12 & 12 & Y \\
2 & S_4 & 6 + 3 & (2A,3C) & 9 & 9 & Y \\
3 & S_4 & 6 + 3 & (2B,3A) & 16 & 25 & N \\
4 & S_4 & 6 + 3 & (2A,3A) & 13 & 13 & Y \\
5 & S_4 & 6 & (2B,3C) & 6 & 6 & Y \\
6 & S_4 & 6 & (2A,3C) & 9 & 9 & Y \\
7 & S_4 & 6 & (2B,3A) & 10 & 13 & N \\
&&&&&& \\
8 & A_5 & 15 & (2B, 3C) & 21 & 21 & Y \\
9 & A_5 & 15 & (2A, 3C) & 21 & 20 & Y \\
10 & A_5 & 15 & (2B, 3A) & 31 & 46 & N \\
11 & A_5 & 15 & (2A, 3A) & 31 & 26 & Y \\
12 & S_5 & 15 + 10 & (2A, 2A) & 41 & 36 & Y \\ 
&&&&&& \\ 
13 & L_3(2) & 21 & (2A, 3C) & 21 & 21 & Y \\
14 & L_3(2) & 21 & (2A, 3A) & 49 & 49 & Y \\
&&&&&& \\
15 & A_6 & 45 & (2A, 3C) & 81 & 70 & Y \\ 
16 & A_6 & 45 & (2A, 3A) & 121 & 76 & Y \\
17 & S_6 & 45 + 15 & (2A, 2B, 3A) & 136 & 91 & Y \\
18 & 3.A_6 & 45 & (2A, 3C, 3C) & 81 & 70 & Y\\
19 & 3.A_6 & 45 & (2A, 3A, 3C) & 141 & 105 & Y\\
20 & 3.A_6 & 45 & (2A, 3A, 3A) & 201 & 76 & Y \\
21 & 3.S_6 & 45 + 45 & (2A, 3A, 3C) & 187 & 136 & Y \\
&&&&&& \\
22 & (S_4 \times S_3) \cap A_7 & 18 + 3 & (2A, 3A) & 34 & 30 & Y \\ 
23 & A_7 & 105 & (2A, 3A)& 406 & 196 & Y\\
24 & S_7 & 105 + 21 & (2A, 2B, 3A) & 427 & 217 & Y \\
24 & 3.A_7 & 105 & (2A, 3A, 3C) & 336 & 211 & Y \\
25 & 3.A_7 & 105 & (2A, 3A, 3A) & 756 & 196 & Y \\
26 & 3.S_7 & 105 + 63 & (2A, 3A, 3C) & 400 & 254 & Y \\
&&&&&& \\
27 & L_2(11) & 55 & (2A, 3A) & 176 & 101 & Y \\ 
28 & L_3(3) & 117 & (2A, 3C, 3A) & 169 & 144 & Y \\ 
29 & M_{11} & 165 & (2A, 3A) & 781 & 286 & Y \\
&&&&&& \\
30 & (S_5 \times S_3) \cap A_8 & 30 + 15 & (3A) & 62 & 67 & N \\
&&&&&& \\
\hline
\end{tabular}
\end{center}
\label{tab:results}
\caption{Certain constructed Majorana representations}
\end{table}

\section{Further work}

Now that we have completely recovered Seress' $2$-closed results, our main aim will be to extend the algorithm so that we are able to construct Majorana representations of higher dimensions. In particular, the next step would be to construct Majorana representations of the groups $A_8$ and $M_{12}$. 

In the case of $M_{12}$, for example, the limiting factor comes in Step 5 of the main algorithm where we are required to solve systems of linear equations over the rationals with approximately 20,000,000 unknown variables. As such, we are particularly focusing on optimising the performance of the linear algebra required by the algorithm. 

In particular, Majorana algebras are defined over the real numbers and the linear algebra in the algorithm is performed over the rational numbers. This is challenging in many respects, not least because we often see an explosion in the coefficients of the matrices involved. However, it may well be of interest to study Majorana algebras defined over finite fields. One possible direction would therefore be to adapt our algorithm to work over generic fields. 

There are also more general aims in Majorana theory which we wish to address with this work. We have already found some new examples of Majorana algebras which are $3$-closed but not $2$-closed. We will be seeking further examples of $3$-closed algebras, as well as algebras which are $4$-closed but not $3$-closed. 

Finally, as mentioned earlier, it is an open problem as to whether there exist non-isomorphic Majorana algebras which share the same shape. If such an example were to exist as a Majorana representation then the algorithm would not be able to automatically construct it. However, our program would help in determining possible candidates for such behaviour and should also be able to construct the required algebras, albeit with some user input.

\end{document}